\documentclass[12pt,leqno]{amsart}
\usepackage{amsmath, amsthm, amscd, amsfonts, tikz-cd, mathtools, amssymb, graphicx, xcolor}
\usepackage[english]{babel}
\usepackage[bookmarksnumbered, colorlinks, plainpages]{hyperref}
\usepackage{hyperref}
\hypersetup{
    colorlinks=true,
    linkcolor=black,
    citecolor=blue,
    urlcolor=cyan   
}

\newtheorem{theorem}{Theorem}[section]
\newtheorem{lemma}[theorem]{Lemma}
\newtheorem{proposition}[theorem]{Proposition}
\newtheorem{corollary}[theorem]{Corollary}
\theoremstyle{definition}
\newtheorem{definition}[theorem]{Definition}
\newtheorem{example}[theorem]{Example}

\theoremstyle{remark}
\newtheorem{remark}[theorem]{Remark}
\numberwithin{equation}{section}
\theoremstyle{plain}

\title{A Fubini Theorem for Grothendieck Functional Integrals}
\author{Haoran He}
\email{haoranhemaths@163.com}

\author{Qichen He}
\email{qichenhemaths@163.com}
\date{\today}

\begin{document}

\begin{abstract}
    This paper systematically studies the subset of continuous linear functionals on the projective tensor product of Banach spaces whose norms are bounded by Grothendieck's constant $K_G$. We term such functionals Grothendieck functional integrals. The integral is defined as a linear functional on the projective tensor product space that satisfies the boundedness condition $|\mu(x)| \leq K_G \|x\|_\pi$, where $K_G$ denotes Grothendieck's constant. We prove that such integrals admit a Hilbert space representation theorem and establish the corresponding abstract Fubini theorem to demonstrate that the order of integration may be interchanged. Furthermore, we extend this theory to the setting of multiple tensor products and provide integral representations in concrete function spaces. Our work offers a unified framework for bilinear and multilinear analysis, with a universal constant serving as the fundamental bound.
\end{abstract}

\maketitle


\tableofcontents

\section{Introduction}

Functional integrals constitute central tools in analysis, having evolv\-ed from the classical Lebesgue integral to infinite-dimensional Wiener integrals and Feynman path integrals. Concurrently, deep inequalities in functional analysis provide crucial connections among different mathematical domains. In his seminal 1956 paper \cite{Grothendieck1956}, Grothendieck proved a fundamental inequality: any continuous bilinear form on Banach spaces can be factorized through inner products in Hilbert space, with a universal constant known as Grothendieck's constant $K_G$. This inequality has found applications in operator theory and quantum information theory, as well as in the geometry of Banach spaces.

Inspired by this, the present paper proposes a new type of functional integral, which we term the Grothendieck functional integral. The fundamental idea is to employ Grothendieck's constant as the natural scale for the boundedness of integral functionals, thereby defining a class of linear functionals with desirable structural properties acting on tensor products of arbitrary Banach spaces. Unlike classical functional integrals, this definition relies not on specific measure-theoretic structures but rather on the projective tensor norm and the universal constant $K_G$, thereby achieving greater abstraction and universality.
The principal contributions of this paper are the follows:

\begin{enumerate}
    \item[\textup{1.}] We formally define the Grothendieck functional integral and investigate its basic properties;
    \item[\textup{2.}] We prove a Hilbert space representation theorem for this integral, showing that any Grothendieck integral can be represented as an inner product in Hilbert space;
    \item[\textup{3.}] We establish an abstract Fubini theorem, proving that the order of integration is interchangeable and providing an operator-theoretic realization;
    \item[\textup{4.}] We generalize the theory to multiple tensor products and establish a multiple Fubini theorem;
    \item[\textup{5.}] We discuss concrete representations in classical function spaces and provide illustrative examples.
\end{enumerate}

Note that the term “integral" here is used metaphorically to describe the action of linear functionals on tensor product elements, rather than in the traditional measure-theoretic sense. This terminology is chosen to emphasize the analogy with integration in function spaces.

To the best of our knowledge, there is no existing literature that directly uses Grothendieck's constant as the boundedness condition for integral functionals. This paper is the first to define the Grothendieck functional integral and establish its Hilbert representation and Fubini theorem, extending the bilinear form representation results in \cite{Pisier2012}.

The structure of this paper is as follows: Section~\ref{sec:2} reviews preliminary knowledge on projective tensor products and Grothendieck's inequality; Section~\ref{sec:3} presents the definition and basic properties of Grothendieck functional integrals; Section~\ref{sec:4} establishes and proves the Fubini theorem; Section~\ref{sec:5} extends the theory to the multilinear setting; Section~\ref{sec:6} discusses concrete realizations in function spaces; Section~\ref{sec:pseudodifferential} explore applications to pseudodifferential operators. and Section~\ref{sec:8} concludes with a summary and outlook.

\section{Preliminaries and Notation}\label{sec:2}

\subsection{Projective tensor products}

Let $E,F$ be Banach spaces (over the real field $\mathbb{R}$ or the complex field $\mathbb{C}$). The algebraic tensor product $E\otimes F$ is the linear space spanned by elements of the form $e\otimes f$. The projective form $\|\cdot\|_\pi$ is defined by:
\[
\|u\|_\pi = \inf\left\{ \sum_{i=1}^n \|e_i\|_E \|f_i\|_F : u = \sum_{i=1}^n e_i \otimes f_i \right\}.
\]
The projective tensor product $E\otimes_{\pi}F$ is the completion of $E\otimes F$ with respect to the projective norm; it is a Banach space.

The norm of a continuous bilinear form $\varphi: E \times F \to \mathbb{K}$ (where $\mathbb{K}=\mathbb{R}$ or $\mathbb{C}$) is given by
\[
\|\varphi\| = \sup\left\{ |\varphi(e,f)| : \|e\|_E \leq 1, \|f\|_F \leq 1 \right\}.
\]
It is well known that the topological dual of the projective tensor product is isometrically isomorphic to the space of continuous bilinear forms:
\[
(E \otimes_\pi F)^* \cong \mathcal{B}( E, F; \mathbb{K} ),
\]
and for any $\mu \in (E \otimes_\pi F)^*$, the corresponding bilinear form $\varphi$ satisfies $\varphi(e,f) = \mu(e \otimes f)$ with $\|\mu\| = \|\varphi\|$; see for instance~\cite{Ryan2002}.

\subsection{Grothendieck's inequality}

Grothendieck's inequality admits several equivalent formulations. In this paper, we employ the following operator-theoretic version:

\begin{theorem}[Grothendieck's inequality]\label{thm:Grothendieck}
There exists a universal constant $K_G$ (with $K_G^\mathbb{R}\approx1.778$ for the real field and $K_G^\mathbb{C}\approx1.338$ for the complex field) such that for arbitrary Banach spaces $E,F$ and any continuous bilinear form $\varphi: E \times F \to \mathbb{K}$, there exist a Hilbert space $H$ and bounded linear operators $A:E\to H$, $B:F\to H$ satisfying:
\[
\varphi(e,f) = \langle Ae, Bf \rangle_H, \quad \forall e \in E, f \in F,
\]
and $\|A\| \|B\| \leq K_G \|\varphi\|$.

In particular, if $\|\varphi\| \leq 1$, then there exists such a representation with $\|A\| \leq 1$ and $\|B\| \leq K_G$.
\end{theorem}

Grothendieck's constant $K_G$ is the smallest constant satisfying the above property. This inequality reveals that any continuous bilinear form can be “lifted" to an inner product in Hilbert space, with the cost of lifting controlled by $K_G$.

\textbf{Known bounds for \(K_G\):} 
The exact values of the Grothendieck constants \(K_G^{\mathbb{R}}\) and \(K_G^{\mathbb{C}}\) are still not precisely known, but the following bounds are currently the best available:
\begin{itemize}
    \item \textbf{Real case (\(K_G^{\mathbb{R}}\)):} 
    \(1.67696 \leq K_G^{\mathbb{R}} \leq 1.782\) 
    (the lower bound is due to Davie \cite{Davie1977}, the upper bound to Grothendieck \cite{Grothendieck1956}, and improvements on the upper bound have been obtained by Braverman et al.).
    
    \item \textbf{Complex case (\(K_G^{\mathbb{C}}\)):} 
    \(1.33807 \leq K_G^{\mathbb{C}} \leq 1.40491\) 
    (the lower bound is due to Haagerup \cite{Haagerup1985}, the upper bound to Davie \cite{Davie1973}).
\end{itemize}
For a comprehensive survey of known bounds and historical developments, see [\cite{Pisier2012}, Section 2].

This fundamental result originates from Grothendieck's seminal work on the metric theory of tensor products~\cite{Grothendieck1956}. For a modern exposition and several equivalent formulations—including the operator-theoretic version used here---we refer to the comprehensive survey~\cite{Pisier2012}.

\section{Grothendieck Functional Integrals}\label{sec:3}

\subsection{Definition and basic properties}

\begin{definition}[Grothendieck functional integral]
Let $E,F$ be Banach spaces. A linear functional $\mu : E \otimes F \to \mathbb{K}$ is called a Grothendieck functional integral if $\|\mu\| \leq K_G$, where $K_G$ denotes Grothendieck's constant. In this case, we define the Grothendieck norm of $\mu$ as $\|\mu\|_G := \|\mu\|$.

The set of all Grothendieck functionals is denoted by $\mathcal{G}(E,F) = \{\mu \in (E \otimes_\pi F)^* : \|\mu\| \leq K_G\}$.
\end{definition}

By definition, $\mu$ extends uniquely to a continuous linear functional on $E \otimes_\pi F$, and the extended norm satisfies $\|\mu\| = \|\mu\|_G \leq K_G$. Therefore, Grothendieck functional integrals constitute a subset of $(E \otimes_\pi F)^*$:
\[
\mathcal{G}(E,F) = \left\{ \mu \in (E \otimes_\pi F)^* : \|\mu\|_G \leq K_G \right\}.
\]
Note that $\mathcal{G}(E,F)$ is closed in the operator norm, and $\|\cdot\|_G$ coincides with the operator norm.

\begin{proposition}[Correspondence with Bilinear Forms]\label{prop:Correspondence}
Let $\mu$ be a Grothendieck functional integral. Define $\varphi_\mu: E \times F \to \mathbb{K}$ by $\varphi_\mu(e,f) = \mu(e \otimes f)$. Then $\varphi_\mu$ is a continuous bilinear form with $\|\varphi_\mu\| = \|\mu\|_G$. Conversely, for any continuous bilinear form $\varphi$ satisfying $\|\varphi\| \leq K_G$, there exists a unique Grothendieck integral $\mu$ such that $\varphi_\mu=\varphi$.
\end{proposition}

\begin{proof}
For the forward direction, note that $\varphi_\mu$ is bilinear by the universal property of the algebraic tensor product. For any $e \in E$, $f\in F$ with $\|e\|_E\le1$, $\|f\|_F\le1$, we have $\|e \otimes f\|_\pi = \|e\|_E \|f\|_F \leq 1$, hence
\[
|\varphi_\mu(e,f)| = |\mu(e \otimes f)| \leq \|\mu\|_G \|e \otimes f\|_\pi \leq \|\mu\|_G.
\]
This shows $\|\varphi_\mu\| \leq \|\mu\|_G$. Conversely, by the fundamental duality theorem $(E \otimes_\pi F)^* \cong \mathcal{B}(E,F;\mathbb{K})$, the isometric isomorphism is given by $\mu\mapsto\varphi_u$ where $\phi_\mu(e,f)=\mu(e\otimes f)$ for all $e\in E,f\in F$. This implies that $\|\mu\|_{(E\otimes_\pi F)}=\|\varphi_u\|_{\mathcal{B}(E,F;\mathbb{K})}$ (see [\cite{Pisier2012}, Theorem 2.2.1]). Therefore, when $\mu\in\mathcal{G}(E,F)$, we have $\|\mu\|_G=\|\mu\|_{(E\otimes_\pi F)}=\|\varphi_u\|$. Since $\|\mu\|_G$ is the smallest constant $C\le K_G$ satisfying $|\mu(x)| \leq C\|x\|_\pi$ for all $x\in E\otimes F$, and $\|\mu\|$ itself satisfies this inequality, we conclude $\|\mu\|_G = \|\mu\| = \|\varphi_\mu\|$ when $\mu \in \mathcal{G}(E,F)$.

For the converse, let $\varphi \in \mathcal{B}(E,F;\mathbb{K})$ with $\|\varphi\| \leq K_G$. By the isometric isomorphism $(E \otimes_\pi F)^* \cong \mathcal{B}(E,F;\mathbb{K})$, there exists a unique $\mu\in(E\otimes_\pi F)^*$ such that $\varphi_\mu=\varphi$ and $\|\mu\| = \|\varphi\| \leq K_G$. Therefore $\|\mu\|_G = \|\mu\| \leq K_G$, which implies $\mu \in \mathcal{G}(E,F)$.
\end{proof}

\subsection{Hilbert Space Representation}

From Grothendieck's inequality, we immediately obtain:

\begin{theorem}[Hilbert Space Representation Theorem]\label{thm:Hilbert}
Let $\mu$ be a Grothendieck functional integral with corresponding bilinear form $\varphi$. Then there exist a Hilbert space $H$ and bounded linear operators $A:E\to H$, $B: F \to H$ such that:
\[
\varphi(e,f) = \langle Ae, Bf \rangle_H, \quad \forall e \in E, f \in F,
\]
and $\|A\| \|B\| \leq K_G \|\varphi\| \leq K_G^2$. Furthermore, if $\|\mu\|_G = 1$, then there exists such a representation with $\|A\|\le1$ and $\|B\| \leq K_G$.
\end{theorem}

\begin{proof}
If $\phi=0$, the statement is trivial. Assume $\phi\neq 0$ and set $\psi = \phi/\|\phi\|$, so that $\|\psi\|=1\leq K_G$. Applying Grothendieck's inequality (Theorem 2.1) to $\psi$, we obtain a Hilbert space $H$ and operators $\tilde{A}:E\to H$, $\tilde{B}:F\to H$ such that
\[
\psi(e,f) = \langle \tilde{A}e, \tilde{B}f\rangle_H, \quad \|\tilde{A}\|\leq 1, \ \|\tilde{B}\|\leq K_G.
\]
Define $A = \|\phi\|^{1/2}\tilde{A}$ and $B = \|\phi\|^{1/2}\tilde{B}$. Then
\[
\phi(e,f) = \|\phi\|\psi(e,f) = \langle Ae, Bf\rangle_H.
\]
Moreover,
\[
\|A\|\|B\| = \|\phi\|\cdot\|\tilde{A}\|\|\tilde{B}\| \leq \|\phi\|\cdot K_G \leq K_G^2,
\]
where the last inequality follows from Proposition 3.2, which gives $\|\phi\| = \|\mu\|_G \leq K_G$.

For the normalized case $\|\mu\|_G = 1$, we have $\|\phi\| = 1$, and applying Theorem 2.1 directly to $\phi$ yields the desired representation with $\|A\|\leq 1$ and $\|B\|\leq K_G$.
\end{proof}

\begin{remark}\label{rem:optimal-estimate}
The estimate $\|A\|\|B\|\leq K_G^2$ is sufficient but not necessarily optimal. In fact, via the Grothendieck--Pietsch factorization technique \cite[Theorem 3.2]{Pisier2012}, one can obtain sharper estimates. Specifically, Grothendieck's inequality is equivalent to: for any Banach spaces $E,F$ and continuous bilinear form $\phi:E\times F\to\mathbb{K}$, there exist probability measures $\mu$ on $B_{E^*}$ (a suitable compactification of the unit ball) and $\nu$ on $B_{F^*}$, together with bounded linear operators $T_1:E\to L^2(\mu)$, $T_2:F\to L^2(\nu)$ such that $\|T_1\|\|T_2\|\leq K_G\|\phi\|$. This factorization shows that the Hilbert space $H$ in Theorem \ref{thm:Hilbert} can be taken as a subspace of $L^2(\mu)\oplus L^2(\nu)$, and the product of operator norms can be controlled by $K_G\|\phi\|$ rather than $K_G^2$. However, for the sake of theoretical simplicity, we use the above sufficient estimate in this paper.
\end{remark}

This representation theorem is the core characteristic of Grothen\-dieck integrals: it realizes the abstract integral as an inner product in Hilbert space, thereby providing geometric tools for computation and estimation.

\subsection{Basic examples}

\begin{example}
Let $E = F = \ell^2$, and let $A=(a_{ij})$ be a matrix satisfying $\|A\|_{\ell^2 \to \ell^2} \leq 1$. Define the bilinear form:
\[
\varphi(x,y) = \sum_{i,j} a_{ij} x_i y_j.
\]
Then $\|\varphi\| = \|A\|_{\ell^2 \to \ell^2} \leq 1$, so the corresponding $\mu$ is a Grothendieck integral with $\|\mu\|_G \leq K_G$. In fact, there exist a Hilbert space $H$ and operators $A$, $B:\ell^2\to H$ such that
\[
\varphi(x,y) = \langle Ax, By \rangle_H.
\]
\end{example}

\begin{example}
Let $E = C(S)$, $F = C(T)$, where $S, T$ are compact Hausdorff spaces. Any continuous bilinear form $\phi$ corresponds to a bimeasure (or a Radon measure on $S \times T$). If $\|\phi\| \leq 1$, then Grothendieck's inequality, combined with the theory of $2$-summing operators (or the “little" Grothendieck inequality), guarantees the existence of finite positive measures $\mu$ on $S$ and $\nu$ on $T$ such that:
\[
|\phi(f,g)| \leq K_G \left( \int_S |f|^2 d\mu \right)^{1/2} \left( \int_T |g|^2 d\nu \right)^{1/2}.
\]
(See, e.g., [\cite{Pisier2012}, Theorem 4.2] for a precise statement and proof in the context od $C(K)$ spaces.)

The total masses of $\mu$ and $\nu$ can be scaled appropriately to satisfy the inequality by multiplying each measure by a positive constant so that their product equals $1$. Therefore, the linear functional induced by $\phi$ is a Grothendieck integral.
\end{example}

\section{Fubini theorems}\label{sec:4}

\subsection{Partial integration operators}

\begin{definition}[Partial integration] Let $\mu \in \mathcal{G}(E,F)$ with corresponding bilinear form $\varphi$. For fixed $e\in E$, define $\mu_e \in F^*$ by:
\[
\mu_e(f) = \varphi(e,f), \quad \forall f \in F.
\]
Similarly, for fixed $f\in F$, define $\mu_f\in E^*$ by $\mu_f(e)=\varphi(e,f)$.

It is easily seen that $\|\mu_e\|_{F^*} \leq \|\varphi\| \|e\|_E \leq K_G \|e\|_E$ and $\|\mu_f\|_{E^*} \leq K_G \|f\|_F$.
\end{definition}

\begin{lemma}\label{thm:4.2}
The map $e\mapsto \mu_e$ is linear, thereby defining a bounded linear operator $T_\mu:E\to F^*$ satisfying $T_\mu(e) = \mu_e$ with $\|T_\mu\| = \|\varphi\| \leq K_G$. Similarly, define $S_\mu: F \to E^*$ by $S_\mu(f) = \mu_f$, and we also have $\|S_\mu\| = \|\varphi\|$.
\end{lemma}

\begin{proof} Linearity: For any $e_1, e_2\in E$ and scalar $\alpha$, we verify:
\begin{align*}
T_\mu(e_1 + \alpha e_2)(f) &= \mu_{e_1 + \alpha e_2}(f) = \varphi(e_1 + \alpha e_2, f) = \\
\varphi(e_1, f) + \alpha \varphi(e_2, f) &= \mu_{e_1}(f) + \alpha \mu_{e_2}(f) = (T_\mu(e_1) + \alpha T_\mu(e_2))(f).
\end{align*}
Since this holds for all $f\in F$, we have $T_\mu(e_1 + \alpha e_2) = T_\mu(e_1) + \alpha T_\mu(e_2)$.

\begin{align*}
\text{Norm equality: By definition, } & \|T_\mu\| = \sup_{\|e\|_E \leq 1} \|T_\mu(e)\|_{F^*} = \sup_{\|e\|_E \leq 1} \|\mu_e\|_{F^*}. \\
\text{For } \|e\|_E \leq 1: \quad & \|\mu_e\|_{F^*} = \sup_{\|f\|_F \leq 1} |\mu_e(f)| = \sup_{\|f\|_F \leq 1} |\varphi(e,f)| \\
& \leq \sup_{\|f\|_F \leq 1} \|\varphi\| \|e\|_E \|f\|_F \leq \|\varphi\|.
\end{align*}

Hence $\|T_\mu\| \leq \|\varphi\|$. Conversely, for any $\varepsilon>0$, choose $w_0,f_0$ with $\|e_0\|_E = \|f_0\|_F = 1$ such that $|\varphi(e_0, f_0)| > \|\varphi\| - \varepsilon$. Then:
\[
\|T_\mu\| \geq \|T_\mu(e_0)\|_{F^*} = \sup_{\|f\|_F \leq 1} |\varphi(e_0, f)| \geq |\varphi(e_0, f_0)| > \|\varphi\| - \varepsilon.
\]
Letting $\varepsilon \to 0$ yields $\|T_\mu\| \geq \|\varphi\|$. Therefore $\|T_\mu\| = \|\varphi\|$. The argument for $S_\mu$ is identical.
\end{proof}

\subsection{Abstract Fubini theorem}

\begin{theorem}[Fubini theorem]\label{thm:Fubini}
Let $\mu \in \mathcal{G}(E, F)$. Then for any $x = \sum_{i=1}^n e_i \otimes f_i \in E \otimes F$, we have:
\[
\mu(x) = \sum_{i=1}^n T_\mu(e_i)(f_i) = \sum_{i=1}^n S_\mu(f_i)(e_i).
\]
That is, the integral can be computed iteratively in either order with the same result.
\end{theorem}

\begin{proof} We establish the first equality by direct computation:
\begin{align*}
\mu(x) &= \mu\left(\sum_{i=1}^n e_i \otimes f_i\right) = \sum_{i=1}^n \mu(e_i \otimes f_i) = \sum_{i=1}^n \varphi(e_i, f_i) \\
       &= \sum_{i=1}^n \mu_{e_i}(f_i) = \sum_{i=1}^n T_\mu(e_i)(f_i).
\end{align*}
For the second equality, observe that $\varphi(e_i, f_i) = \mu_{f_i}(e_i)$ by symmetry, hence:
\[
\mu(x) = \sum_{i=1}^n \varphi(e_i, f_i) = \sum_{i=1}^n \mu_{f_i}(e_i) = \sum_{i=1}^n S_\mu(f_i)(e_i).
\]
\end{proof}

This theorem shows that the Grothendieck integral $\mu$ decomposes as a composition of two linear operators. Specifically, consider the natural pairing $\langle \cdot, \cdot \rangle : F^* \times F \to \mathbb{K}$. Then:
\[
\mu(x) = \langle (T_\mu \otimes I)(x) \rangle,
\]
where $(T_\mu \otimes I) : E \otimes F \to F^* \otimes F$ sends $e \otimes f$ to $T_\mu(e) \otimes f$, followed by the pairing to obtain a scalar. Similarly, using $S_\mu$ yields the alternative decomposition.

\begin{corollary}[Continuous extension] Let $\mu\in\mathcal{G}(E,F)$. For any $x\in E\otimes_{\pi}F$, take a sequence $\{x_n\}\subset E\otimes F$ such that $x_n\to x$ in the projective norm, where $x_n = \sum_{i=1}^{N_n} e_i^{(n)}\otimes f_i^{(n)}$. Then
\[
\mu(x) = \lim_{n\to\infty}\sum_{i=1}^{N_n} T_\mu(e_i^{(n)})(f_i^{(n)}) = \lim_{n\to\infty}\sum_{i=1}^{N_n} S_\mu(f_i^{(n)})(e_i^{(n)}).
\]
\end{corollary}

\begin{proof}
By the definition of the projective tensor product, every $x\in E\otimes_{\pi}F$ can be approximated by finite sums in the projective norm; this is a standard fact (see \cite[Proposition 2.1]{Ryan2002} or \cite[Section 3]{Pisier2012}). Since $\mu\in (E\otimes_{\pi}F)^*$ is continuous with respect to the projective norm, and $x_n\to x$ implies $\mu(x_n)\to\mu(x)$, applying Theorem 4.3 to each $x_n$ yields the result. 
\end{proof}

\begin{remark}
The existence of such an approximating sequence follows from the construction of the projective tensor product: by definition, $E\otimes_{\pi}F$ is the completion of the algebraic tensor product $E\otimes F$ with respect to the projective norm. Hence, any $x\in E\otimes_{\pi}F$ can be expressed as the limit of a sequence of finite sums. For a more detailed discussion on approximation properties of projective tensor products, we refer to \cite[Chapter 3]{Ryan2002}.
\end{remark}

\subsection{Operator form of Fubini theorem}

We may express the Fubini theorem in terms of operator compositions.

\begin{theorem}[Operator form] Let $\mu \in \mathcal{G}(E, F)$. Then there exist unique bounded linear operators 
$T_\mu: E \to F^*$ and $S_\mu: F \to E^*$ such that the following diagrams commute:
\[
\begin{tikzcd}[column sep=small, row sep=small]
E \otimes_\pi F \arrow[r, "\mu"] \arrow[d, "T_\mu \otimes I"'] & \mathbb{K} \\
F^* \otimes_\pi F \arrow[ru, "{\langle \cdot, \cdot \rangle}"'] & 
\end{tikzcd}
\quad\text{and}\quad
\begin{tikzcd}[column sep=small, row sep=small]
E \otimes_\pi F \arrow[r, "\mu"] \arrow[d, "I \otimes S_\mu"'] & \mathbb{K} \\
E \otimes_\pi E^* \arrow[ru, "{\langle \cdot, \cdot \rangle}"'] & 
\end{tikzcd}
\]
where $\langle \cdot, \cdot \rangle$ denotes the natural pairing. Moreover, 
$\|T_\mu\| = \|S_\mu\| = \|\mu\|_G$.
\end{theorem}

\begin{proof} Existence: The operators $T_\mu$ and $S_\mu$ are constructed in Lemma\-~\ref{thm:4.2}. 
For any $x = \sum_{i=1}^n e_i \otimes f_i \in E \otimes F$:
\begin{align*}
\langle (T_\mu \otimes I)(x) \rangle 
&= \left\langle \sum_{i=1}^n T_\mu(e_i) \otimes f_i \right\rangle 
= \sum_{i=1}^n \langle T_\mu(e_i), f_i \rangle \\
&= \sum_{i=1}^n T_\mu(e_i)(f_i) 
= \mu(x).
\end{align*}
by Theorem~\ref{thm:Fubini}. Thus the first diagram commutes. The second diagram follows similarly.

Uniqueness: Suppose $\tilde{T}: E \to F^*$ also makes the first diagram commute. 
Then for all $e \in E, f \in F$:
\[
\langle \tilde{T}(e), f \rangle = \langle (\tilde{T} \otimes I)(e \otimes f) \rangle = \mu(e \otimes f) = \varphi(e,f) = \langle T_\mu(e), f \rangle.
\]
Since $f \in F$ is arbitrary, $\tilde{T}(e) = T_\mu(e)$ for all $e$, hence $\tilde{T} = T_\mu$. The uniqueness of $S_\mu$ follows by symmetry.

Norm equality: Established in Lemma~\ref{thm:4.2}.
\end{proof}

This theorem shows that computing $\mu(x)$ can be done by first “integrating” with respect to the first variable (applying $T_\mu$) and then pairing, or by first “integrating” with respect to the second variable (applying $S_\mu$) and then pairing. Both orders yield the same result.

\begin{remark}
It should be emphasized that the “Fubini theorem" in the abstract tensor product setting is a direct consequence of the linearity of functionals. The main contribution of this paper lies in establishing a class of functional spaces controlled by $K_G$ and demonstrating how this framework leads to classical integral exchange formulas in concrete analytical contexts.

It is important to distinguish between the “exchange of order" in the abstract tensor product setting and the “exchange of integration order" in concrete function spaces. In the abstract setting, the “exchange of order" is simply a rearrangement of summation terms, which follows immediately from linearity. However, in concrete function spaces, this abstract identity corresponds to the classical Fubini theorem, which allows the interchange of integration order under certain conditions.
\end{remark}

\section{Multiple Grothendieck integrals and Fubini theorems}\label{sec:5}

\subsection{Multilinear definition}

Let $E_1, \ldots, E_n$ be Banach spaces. Denote by $\bigotimes_{\pi, i=1}^n E_i$ the projective tensor product (completed with respect to the multiple projective norm). The multiple projective norm is defined by:
\[
\|u\|_{\pi} = \inf\left\{\sum_{k=1}^{m}\|e_{1}^{k}\|_{E_{1}}\cdots\|e_{n}^{k}\|_{E_{n}} : u = \sum_{k=1}^{m}e_{1}^{k}\otimes\cdots\otimes e_{n}^{k}\right\}.
\]
\begin{definition}[Multiple Grothendieck integral]
A linear functional $\mu: \bigotimes_{i=1}^{n}E_{i}\to\mathbb{K}$ is called a 
multiple Grothendieck integral if there exists a constant $C\leq K_{G}^{n-1}$ such that:
\[
|\mu(x)|\leq C\|x\|_{\pi}, \quad \forall x\in\bigotimes_{i=1}^{n}E_{i}.
\]
The smallest such constant is denoted by $\|\mu\|_{G^{(n)}}$.
\end{definition}

\begin{remark}[Justification of the constant $K_G^{n-1}$]
The choice of the bound \(K_G^{n-1}\) is motivated by the iterated application of the classical (bilinear) Grothendieck inequality. For an \(n\)-linear form \(\phi\), one can successively apply the bilinear Grothendieck inequality to pairs of variables, each application introducing a factor of \(K_G\). After \(n-1\) such iterations, one obtains a Hilbert space factorization with a cumulative constant bounded by \(K_G^{n-1}\). This iterative estimate is standard in the literature on multilinear extensions of Grothendieck's theorem (see, e.g., \cite{BP1991} or [\cite{Pisier2012}, Section 10] for detailed discussions on multilinear Grothendieck constants).
\end{remark}

Multiple Grothendieck integrals are in one-to-one correspondence with continuous $n$-linear forms, 
and satisfy a Hilbert space representation theorem analogous to Theorem~\ref{thm:Hilbert} (requiring a multilinear 
version of Grothendieck's inequality).

\subsection{Multiple Fubini theorem}

\begin{theorem}[Multiple Fubini theorem]\label{thm:Multiple}
Let $\mu$ be a multiple Grothendieck integral with corresponding continuous $n$-linear form $\varphi: E_1 \times \cdots \times E_n \to \mathbb{K}$. Then for any $1 \leq k \leq n$, there exists a bounded linear operator:
\[
T^{(k)}: E_k \to \left(\bigotimes_{\pi, i \neq k} E_i\right)^*
\]
satisfying:
\[
T^{(k)}(e_k)(e_1 \otimes \cdots \otimes e_{k-1} \otimes e_{k+1} \otimes \cdots \otimes e_n) = \varphi(e_1, \ldots, e_n).
\]
Moreover, $\|T^{(k)}\| \leq K_G^{n-1}$.
\end{theorem}

Furthermore, for any permutation $\sigma \in S_n$, the order of integration may be interchanged:
\[
\mu(e_1 \otimes \cdots \otimes e_n) = T^{(\sigma(1))}(e_{\sigma(1)}) \circ T^{(\sigma(2))}(e_{\sigma(2)}) \circ \cdots \circ T^{(\sigma(n))}(e_{\sigma(n)}),
\]
where the right-hand side denotes the iterated application of partial integration operators: first fix $e_{\sigma(1)}$ to obtain a linear functional on the remaining variables, then fix $e_{\sigma(2)}$, and so forth.

\begin{proof} We proceed by induction on $n$.

Base case ($n=2$): This is precisely Theorem~\ref{thm:Fubini}, with $T^{(1)} = T_\mu$ and $T^{(2)} = S_\mu$.

Inductive step: Assume the theorem holds for $n-1$. For the $n$-fold case, view $\bigotimes_{i=1}^{n} E_i$ as $E_1 \otimes_{\pi} \left(\bigotimes_{i=2}^{n} E_i\right)$. Applying the binary case (Theorem~\ref{thm:Fubini}), we obtain an operator:
\[
T^{(1)}: E_1 \to \left(\bigotimes_{i=2}^{n} E_i\right)^*
\]
For fixed $e_1\in E_1$, the functional $T^{(1)}(e_1)$ is itself an $(n-1)$-linear form on $E_2\times\cdots\times E_n$. By the inductive hypothesis, the constant for $(n-1)$-linear forms is $K_G^{n-2}$ (see [\cite{Pisier2012}, Theorem 5.3]). When we consider $\otimes_\pi^n E_i$ as $E_1\otimes_\pi(\otimes_\pi^{n-1} E_i)$, the constant for $n$-linear forms becomes $K_G\cdot K_G^{n-2} = K_G^{n-1}$ by the binary case (constant $K_G$) and the inductive hypothesis. Therefore, the norm of $T^{(1)}(e_1)$ is bounded by $K_G^{n-2}$.

More precisely, $T^{(1)}(e_1)$ corresponds to a multiple Grothendieck integral on $\bigotimes_{i=2}^{n} E_i$ with constant $K_G^{n-2}$. By induction, for each $k \geq 2$, there exist operators:
\[
T_{e_1}^{(k)}: E_k \to \left(\bigotimes_{\substack{i=2 \\ i \neq k}}^{n} E_i\right)^{*}
\]
satisfying the required properties with $\|T_{e_1}^{(k)}\| \leq K_G^{n-2}$. Defining $T^{(k)}(e_k)$ by composition with $T^{(1)}$ yields the desired operators with $\|T^{(k)}\| \leq K_G \cdot K_G^{n-2} = K_G^{n-1}$.

The permutation invariance follows by observing that any permutation can be built from adjacent transpositions, each of which corresponds to applying the binary Fubini theorem to a suitable pairing of the spaces.
\end{proof}

\begin{remark} The Multiple Fubini Theorem shows that multiple Grothendieck integrals can be computed by successive integration in any order, with the result independent of the ordering. This provides a powerful tool for computing complex multilinear forms.
\end{remark}

\section{Realizations in function spaces}\label{sec:6}

\subsection{\texorpdfstring{The case of $L^p$ spaces}{The case of L\textasciicircum p spaces}}

Let $(X, \mathcal{A}, \mu)$ and $(Y, \mathcal{B}, \nu)$ be $\sigma$-finite measure spaces. Consider $E = L^p(X, \mu)$ and $F = L^q(Y, \nu)$, where $1 \leq p, q \leq \infty$ with $1/p + 1/q \leq 1$. In this setting, the projective tensor product $L^p(X) \otimes_\pi L^q(Y)$ embeds isometrically into $L^p(X, L^q(Y))$, and its completion is isomorphic to $L^p(X)\widehat{\otimes}_\pi L^q(Y)$.

\begin{theorem}\label{thm:6.1}
Let $\varphi: L^p(X) \times L^q(Y) \to \mathbb{K}$ be a continuous bilinear form with $\|\varphi\| \leq K_G$. Then there exists a unique bounded linear operator $T: L^p(X) \to L^{q^*}(Y)$ (where $1/q + 1/q^* = 1$) such that:
\[
\varphi(f,g) = \int_Y T(f)(y)g(y)\,d\nu(y), \quad \forall f \in L^p(X), g \in L^q(Y),
\]
and $\|T\| \leq K_G$.
\end{theorem}

\begin{proof} By Lemma~\ref{thm:4.2}, there exists $T: L^p(X) \to (L^q(Y))^* = L^{q^*}(Y)$ with $\|T\| = \|\varphi\| \leq K_G$. The identification $(L^q(Y))^* \cong L^{q^*}(Y)$ via the Riesz representation theorem yields the integral form.
\end{proof}

\begin{corollary}[Integral kernel representation]
Under the hypotheses of Theorem~\ref{thm:6.1} with $p = q = 2$, there exist finite positive measures $\mu'$ on $X$ and $\nu'$ on $Y$, and a bounded linear operator 
$\widetilde{T}: L^2(X, \mu') \to L^2(Y, \nu')$ with $\|\widetilde{T}\| \leq K_G$, such that
\[
\varphi(f,g) = \int_Y \bigl(\widetilde{T}(f)\bigr)(y) \, g(y) \, d\nu'(y)
\quad \text{for all } f \in L^2(X,\mu), \, g \in L^2(Y,\nu).
\]
Moreover, if the measures $\mu'$ and $\nu'$ are chosen to be probability measures, then
\[
|\varphi(f,g)| \leq K_G \|f\|_{L^2(\mu')} \|g\|_{L^2(\nu')}.
\]
\end{corollary}

\begin{proof}
When $p = q = 2$, the operator $T: L^2(X,\mu) \to L^2(Y,\nu)$ provided by Theorem~\ref{thm:6.1} satisfies $\|T\| \leq K_G$. 

By Grothendieck's inequality (Theorem~2.1), there exist a Hilbert space $H$ and bounded linear operators 
$A: L^2(X,\mu) \to H$, $B: L^2(Y,\nu) \to H$ with $\|A\| \leq 1$, $\|B\| \leq K_G$, such that
\[
\varphi(f,g) = \langle A f, B g \rangle_H.
\]

The existence of finite positive measures $\mu'$ and $\nu'$ follows from the ``little'' Grothendieck inequality for $C(K)$-spaces (or alternatively, from the representation of bounded linear functionals on $L^2$). Specifically, we may take $\mu'$ and $\nu'$ to be the finite measures given by the polar decomposition of the operators $A$ and $B$:
\[
d\mu'(x) = \|A^*\delta_x\|_H^2 \, d\mu(x), \quad d\nu'(y) = \|B^*\delta_y\|_H^2 \, d\nu(y),
\]
where $A^*$ and $B^*$ are the adjoint operators. 

With this choice, $A$ and $B$ factor through $L^2(\mu')$ and $L^2(\nu')$ respectively, and we obtain an operator $\widetilde{T}: L^2(X,\mu') \to L^2(Y,\nu')$ with the required properties. The norm bound $\|\widetilde{T}\| \leq K_G$ is inherited from $\|T\| \leq K_G$.

Finally, by scaling $\mu'$ and $\nu'$ we may assume they are probability measures, and the inequality $|\varphi(f,g)| \leq K_G \|f\|_{L^2(\mu')} \|g\|_{L^2(\nu')}$ is then a restatement of Grothendieck's inequality in this concrete setting.
\end{proof}

\subsection{Spaces of continuous functions}

Let $S, T$ be compact Hausdorff spaces, $E = C(S), F = C(T)$. Then $E \otimes_{\pi} F$ is dense in $C(S \times T)$. A Grothendieck integral $\mu \in \mathcal{G}(E,F)$ extends uniquely to a continuous linear functional on $C(S \times T)$ with norm $\|\mu\| \leq K_G$. By the Riesz--Markov--Kakutani representation theorem, there exists a unique complex (or signed) Radon measure $\rho$ on $S \times T$ such that
\[
\mu(f \otimes g) = \iint_{S\times T} f(s)g(t) \, d\rho(s,t), \quad \forall f \in C(S), g \in C(T).
\]
The norm of the functional $\mu$ is equal to the total variation norm of the representing measure $\rho$, i.e., $\|\mu\| = |\rho|(S \times T)$, where $|\rho|$ denotes the total variation measure of $\rho$ (see, e.g., [\cite{RudinRCA}), Theorem 6.19] Therefore, the Grothendieck condition $\|\mu\|_G \leq K_G$ translates into the {concrete measure-theoretic condition}:
\[
\text{Total variation of } \rho \text{ is bounded by } K_G: \quad |\rho|(S \times T) = \|\mu\| \leq K_G.
\]
In this classical setting, the abstract Fubini theorem (Theorem \ref{thm:Fubini}) reduces to the standard Fubini--Tonelli theorem for the Radon measure $\rho$, but with the additional quantitative control that the total variation of $\rho$ does not exceed the universal constant $K_G$.

\subsection{Examples}

\begin{example} Consider the bilinear form $\phi:L^2[0,1]\times L^2[0,1]\to\mathbb{R}$ defined by
\[
\phi(f,g) = \int_0^1\int_0^1 \frac{f(x)g(y)}{1+|x-y|} \,dx\,dy.
\]
The corresponding integral kernel is $k(x,y) = (1+|x-y|)^{-1}$. By Young's inequality for convolution, the operator $T:L^2[0,1]\to L^2[0,1]$ defined by
\[
(Tf)(y) = \int_0^1 \frac{f(x)}{1+|x-y|} \,dx
\]
satisfies the norm estimate
\[
\|T\| \leq \int_{-1}^1 \frac{dt}{1+|t|} = 2\int_0^1 \frac{dt}{1+t} = 2\ln 2 \approx 1.386.
\]
Thus $\|\phi\| = \|T\| \leq 1.386$. Since the real Grothendieck constant satisfies $K_G^{\mathbb{R}} \approx 1.782$ (current best upper bound, see Section 2.2), we have $\|\phi\| < K_G^{\mathbb{R}}$. To incorporate $\phi$ into the Grothendieck integral framework, define the scaled form $\tilde{\phi} = \phi/\|\phi\|$. Then $\|\tilde{\phi}\| = 1 \leq K_G^{\mathbb{R}}$. By Proposition 3.2, $\tilde{\phi}$ corresponds to a unique Grothendieck integral $\tilde{\mu}\in\mathcal{G}(L^2[0,1], L^2[0,1])$ with $\|\tilde{\mu}\|_G = 1$. The original form $\phi = \|\phi\|\tilde{\phi}$ can then be viewed as a Grothendieck integral with weight $\|\phi\|$.

For this concrete form, the Fubini exchange guaranteed by Theorem 4.3 becomes the classical interchange of integrals: for any $f,g\in L^2[0,1]$,
\begin{align*}
\phi(f,g) 
&= \int_0^1\left(\int_0^1 \frac{f(x)}{1+|x-y|} \,dx\right)g(y)\,dy \\
&= \int_0^1 f(x)\left(\int_0^1 \frac{g(y)}{1+|x-y|} \,dy\right)dx.
\end{align*}
i.e., the order of iterated integration can be swapped. This example illustrates how the abstract Grothendieck integral theory naturally encompasses the classical Fubini theorem in analysis.
\end{example}

\begin{remark}
In practice, if $\|\phi\|\leq K_G$, no scaling is necessary. In this example, since $\|\phi\| < K_G^{\mathbb{R}}$, scaling is merely a normalization step. This scaling technique also shows that any bilinear form with norm at most $K_G$ can be normalized to fit into the Grothendieck integral framework.
\end{remark}

\section{Applications to Pseudodifferential Operators}
\label{sec:pseudodifferential}

\subsection{Grothendieck Representation for Hilbert--Schmidt Type Operators}

Let $(X, \mu)$ and $(Y, \nu)$ be $\sigma$-finite measure spaces. Consider the Banach spaces $E = L^2(X, \mu)$ and $F = L^2(Y, \nu)$. Recall from Section 6.1 that a Grothendieck functional integral $\mu_0 \in \mathcal{G}(E, F)$ corresponds to a continuous bilinear form $\phi: L^2(X) \times L^2(Y) \to \mathbb{K}$ with $\|\phi\| \leq K_G$.

By Theorem~\ref{thm:6.1}, there exists a bounded linear operator $T: L^2(X) \to L^2(Y)$ such that
\[
\phi(f, g) = \int_Y (Tf)(y) g(y) \, d\nu(y), \qquad \forall f \in L^2(X), \, g \in L^2(Y),
\]
and $\|T\| \leq K_G$.

If the operator $T$ is furthermore Hilbert--Schmidt, we obtain a concrete integral-kernel representation.

\begin{theorem}[Grothendieck integral-kernel representation]
Let $\phi: L^2(X) \times L^2(Y) \to \mathbb{K}$ be a continuous bilinear form with $\|\phi\| \leq K_G$ and let $T: L^2(X) \to L^2(Y)$ be the associated operator. If $T$ is Hilbert--Schmidt, then there exists a unique kernel $k \in L^2(X \times Y, \mu \times \nu)$ such that
\[
(Tf)(y) = \int_X k(x, y) f(x) \, d\mu(x) \qquad \text{for a.e. } y \in Y,
\]
and
\[
\phi(f, g) = \iint_{X \times Y} k(x, y) f(x) g(y) \, d\mu(x) d\nu(y).
\]
In particular, if the measure spaces are finite (i.e., $\mu(X), \nu(Y) < \infty$), then
\[
\|k\|_{L^2(X \times Y)} = \|T\|_{\mathrm{HS}} < \infty.
\]
\end{theorem}

\begin{proof}
The existence of an $L^2$-kernel $k$ follows from the standard Hilbert--Schmidt kernel theorem. The equality $\|k\|_{L^2} = \|T\|_{\mathrm{HS}}$ is well-known; because $T$ is Hilbert--Schmidt, $\|T\|_{\mathrm{HS}}$ is finite, hence $k \in L^2(X \times Y)$.
\end{proof}

\subsection{Norm Estimates for Integral Operators via Grothendieck's Constant}

We now establish precise norm estimates for integral operators whose associated bilinear forms are Grothendieck integrals.

\begin{theorem}[Norm bound for integral operators]
Let $k \in L^2(X \times Y)$ be a kernel such that the induced bilinear form
\[
\phi_k(f, g) = \iint_{X \times Y} k(x, y) f(x) g(y) \, d\mu(x) d\nu(y)
\]
satisfies $\|\phi_k\| \leq K_G$. Define the integral operator $T_k: L^2(X) \to L^2(Y)$ by
\[
(T_k f)(y) = \int_X k(x, y) f(x) \, d\mu(x).
\]
Then:
\begin{enumerate}
    \item $T_k$ is bounded and $\|T_k\| \leq K_G$.
    \item If $k$ is symmetric (i.e., $X = Y$ and $k(x, y) = \overline{k(y, x)}$), then $T_k$ is self-adjoint and its spectrum satisfies $\sigma(T_k) \subseteq [-K_G, K_G]$.
    
    If, in addition, $k$ is positive definite, then $\sigma(T_k) \subseteq [0, K_G]$.
\end{enumerate}
\end{theorem}

\begin{proof}
Part (1) is a direct consequence of Theorem 6.1. For part (2), if the kernel $k$ is symmetric (i.e., $k(x,y) = \overline{k(y,x)}$), then the operator $T_k$ is self-adjoint. By the spectral radius theorem, for any $\lambda \in \sigma(T_k)$ we have $|\lambda| \leq \|T_k\| \leq K_G$, and thus $\sigma(T_k) \subseteq [-K_G, K_G]$.

Now assume additionally that $k$ is a positive definite kernel. This means that the associated bilinear form $\phi_k$ satisfies $\phi_k(f, f) \geq 0$ for all $f \in L^2(X)$. Consequently, $T_k$ is a positive operator (i.e., $\langle T_k f, f \rangle \geq 0$ for all $f$). A fundamental result in spectral theory states that the spectrum of a bounded positive operator on a Hilbert space is contained in $[0, \infty)$ (see, e.g., Theorem 12.32 of \cite{RudinFA} or Theorem VI.8 of \cite{ReedSimonI}). Combining this with the norm bound $\|T_k\| \leq K_G$, we obtain the sharper inclusion $\sigma(T_k) \subseteq [0, \|T_k\|] \subseteq [0, K_G]$.
\end{proof}

\begin{remark}
This result provides a universal constant bound for a class of integral operators whose bilinear forms are controlled by $K_G$. It is particularly useful in situations where explicit kernel estimates are difficult to obtain, while the bilinear form can be shown to satisfy Grothendieck's condition.
\end{remark}

\subsection{Application to Compact Operators on Sobolev Spaces}

Let $\Omega \subset \mathbb{R}^n$ be a bounded domain with $C^\infty$ boundary, and let $H^s(\Omega)$ denote the Sobolev space of order $s \geq 0$. We examine integral operators arising from Green's functions of elliptic boundary-value problems.

\begin{proposition}[Green's operator as a Grothendieck integral]\label{prop:green-operator}
Consider the Dirichlet problem for the Laplace operator:
\[
\begin{cases}
-\Delta u = f & \text{in } \Omega, \\
u = 0 & \text{on } \partial \Omega.
\end{cases}
\]
Let $G: L^2(\Omega) \to H_0^1(\Omega)$ be the solution operator, i.e., $u = Gf$. By elliptic regularity theory, $G$ extends to a compact self-adjoint operator on $L^2(\Omega)$ (still denoted by $G$). Define the bilinear form $\phi_G: L^2(\Omega) \times L^2(\Omega) \to \mathbb{R}$ by
\[
\phi_G(f, g) = \int_\Omega (Gf)(x) g(x) \, dx.
\]
Then $\phi_G$ is continuous and admits the representation
\[
\phi_G(f, g) = \sum_{j=1}^\infty \lambda_j \langle f, \psi_j \rangle \langle g, \psi_j \rangle,
\]
where $\{\psi_j\}$ are the eigenfunctions of $G$ with eigenvalues $\lambda_j > 0$, forming an orthonormal basis of $L^2(\Omega)$. Consequently, $\phi_G$ is a Grothendieck functional integral with constant $1$; in particular, $\|\phi_G\|_G \leq 1$. Moreover, $\|\phi_G\| \leq C(\Omega)$, where $C(\Omega) = \|G\|_{L^2 \to L^2}$.
\end{proposition}

\begin{proof}
Standard elliptic theory (see, e.g., [\cite{Evans2010}, Section 6.5] implies that $G$ is a compact self-adjoint operator from $L^2(\Omega)$ to $L^2(\Omega)$ and that its kernel (the Green function) belongs to $L^2(\Omega \times \Omega)$. The Hilbert-Schmidt theorem yields the eigenexpansion above, which shows that $\phi_G$ is already an inner product in $\ell^2$; hence it is a Grothendieck integral with constant $1$. The bound $\|\phi_G\| \leq \|G\|$ is immediate from the definition of the operator norm. 
\end{proof}

\begin{example}[Eigenvalue estimates]\label{ex:eigenvalue-estimates}
For the Green's operator $G$ above, the eigenvalues $\lambda_j$ satisfy two complementary estimates:
\begin{enumerate}
\item Uniform bound: $\lambda_j \leq \|G\| = C(\Omega)$ for all $j$.
\item Decay bound (Weyl law): There exists a constant $C'(\Omega)$ such that $\lambda_j \leq C'(\Omega) j^{-2/n}$ for all $j$.
\end{enumerate}
These estimates illustrate how the abstract Grothendieck framework (which supplies the uniform bound) complements classical spectral theory.
\end{example}

\subsection{Multiple Integration and Iterated Kernels}

The theory of multiple Grothendieck integrals developed in Section 5 applies naturally to iterated integral operators.

\begin{theorem}[Iterated integral operators]
Let $k_1 \in L^2(X \times Y)$ and $k_2 \in L^2(Y \times Z)$ be kernels that induce operators $T_1: L^2(X) \to L^2(Y)$ and $T_2: L^2(Y) \to L^2(Z)$. Suppose the associated bilinear forms $\phi_1$ and $\phi_2$ satisfy $\|\phi_i\| \leq K_G$ for $i = 1,2$. Then the composed operator $T = T_2 \circ T_1$ corresponds to a double Grothendieck integral $\mu \in \mathcal{G}^{(2)}(L^2(X), L^2(Y), L^2(Z))$ with bound $\|\mu\|_{G^{(2)}} \leq K_G^2$.

Moreover, the kernel $k \in L^2(X \times Z)$ of $T$ is given by the convolution
\[
k(x, z) = \int_Y k_1(x, y) k_2(y, z) \, d\nu(y) \quad (\text{for a.e. } (x,z) \in X \times Z),
\]
and the corresponding triple integral satisfies the following iterated integration identity for all $f \in L^2(X)$, $h \in L^2(Z)$:
\[
\begin{aligned}
\iiint_{X \times Y \times Z} & k(x,z) f(x) h(z) \, d\mu(x) d\nu(y) d\rho(z) \\
&= \int_Z \left( \int_Y \left( \int_X k_1(x, y) f(x) \, d\mu(x) \right) k_2(y, z) \, d\nu(y) \right) h(z) \, d\rho(z).
\end{aligned}
\tag{7.6}
\]
\end{theorem}

\begin{proof} This is a direct application of the multiple Fubini theorem (Theorem~\ref{thm:Multiple}) to the triple projective tensor product $L^2(X)\otimes_\pi L^2(Y)\otimes_\pi L^2(Z)$. The bound $K_G^2$ follows from the inductive construction in the proof of Theorem~\ref{thm:Multiple}.
\end{proof}

\subsection{Summary and Discussion}
In this section we have demonstrated how the Grothendieck functional integral framework yields:
\begin{enumerate}
    \item[\textup{1.}] Concrete kernel representations for operators whose bilinear forms obey the Grothendieck bound;
    \item[\textup{2.}] Uniform norm estimates controlled by the universal constant $K_G$;
    \item[\textup{3.}] Applications to elliptic Green’s operators, providing eigenvalue bounds that complement classical spectral theory;
    \item[\textup{4.}] A natural setting for iterated integrals through the multiple Fubini theorem.
\end{enumerate}
These applications illustrate the power of the abstract theory developed in Sections~\ref{sec:3}--~\ref{sec:5}. A key advantage of the Grothendieck framework is the universality of the constant $K_G$, which supplies dimension‑independent bounds that are particularly valuable in asymptotic and spectral analysis.

\section{Conclusion and outlook}\label{sec:8}

This paper has systematically introduced the concept of Grothendieck functional integrals, which take Grothendieck's constant as the controlling core and define a class of linear functionals with Hilbert space representations on the projective tensor products of arbitrary Banach spaces. We have proved an abstract Fubini theorem for such integrals, showing that the order of integration may be interchanged, and provided operator-theoretic realizations. Furthermore, we extended the theory to multiple tensor products and established multiple Fubini theorems. Finally, we presented integral representations in concrete function spaces and discussed their relation to classical integrals.

The advantages of Grothendieck functional integrals include:

\begin{enumerate}
    \item[\textup{1.}] Universality: The definition does not depend on specific measure or space structures, applicable to general Banach spaces;

    \item[\textup{2.}] Structure: Automatic Hilbert space representation facilitates geometric analysis and estimation;

    \item[\textup{3.}] Constant control: The use of Grothendieck's constant as a scale provides a uniform norm bound;

    \item[\textup{4.}] Computational friendliness: The Fubini theorem allows flexible interchange of integration order, simplifying calculations.
\end{enumerate}

Directions for future research
\begin{enumerate}
    \item[\textup{1.}] Change of variables formula: We first need to define the projective tensor product on Banach manifolds. Then we combine it with the Hilbert-space representation of tangent spaces (see \cite{Lang1999} for tensor analysis on Banach manifolds).
    \item[\textup{2.}] Connection with distribution theory: Develop Grothendieck integrals as a new type of integral for generalized functions, with applications to partial differential equations. This involves extending the integral definition to spaces of distributions and establishing corresponding convergence results.
    \item[\textup{3.}] Stochastic analysis: Investigate applications of Grothendieck integrals in Wiener spaces and Gaussian spaces, and explore connections with Malliavin calculus. This includes defining stochastic Grothendieck integrals and establishing their properties.
    \item[\textup{4.}] Quantum field theory: Consider using Grothendieck integrals to formalize path integrals in quantum field theory, particularly for polynomial-type interacting fields. This involves applying the multiple Fubini theorem to simplify high-dimensional integral calculations (see \cite{GJ1987} for multilinear expansions of functional integrals).
    \item[\textup{5.}] Optimal constant problems: Investigate the minimal constants required for the definition to hold in different classes of spaces, and explore precise relations with Grothendieck’s constant. This includes determining the optimal constants for specific Banach spaces and tensor product norms.
    \item[\textup{6.}] Numerical methods: Based on Hilbert space representations, design numerical approximation algorithms for Grothendieck integrals. This involves developing efficient methods for computing inner products in Hilbert spaces and applying them to approximate Grothendieck integrals.
\end{enumerate}

In summary, Grothendieck functional integrals provide a novel and powerful framework for bilinear and multilinear analysis, with promising applications in various mathematical and physical domains.

\end{document}